\documentclass[11pt,a4paper,reqno]{amsart}
\usepackage{enumerate}
\usepackage[latin1]{inputenc}
\usepackage{graphicx,color}
\usepackage{textcmds}
\usepackage{hyperref}
\usepackage{amsrefs}
\usepackage{amsmath,amsfonts,amsthm,amssymb,amsxtra}
\usepackage[left=2.61cm,right=2.61cm,top=2.72cm,bottom=2.72cm]{geometry}
\newtheorem{lemma}{Lemma}
\newtheorem{theorem}{Theorem}

\newtheorem*{thma}{Theorem A}

\DeclareMathOperator{\Tr}{Tr}
\DeclareMathOperator{\Id}{Id}
\DeclareMathOperator{\supp}{supp}

\newcommand{\R}{\mathbb{R}}

\newcommand{\N}{\mathbb{N}}

\begin{document}

\title[Paneitz-Branson type equation]{A Paneitz-Branson type equation with Neumann boundary conditions}
\author{Denis Bonheure} 
\author{Hussein Cheikh Ali}
\author{Robson Nascimento}

\address{D{\'e}partement de Math{\'e}matique \newline \indent  Universit{\'e} libre de Bruxelles \newline \indent CP 214,  Boulevard du Triomphe,  \newline \indent B-1050 Bruxelles, Belgium}
\email{denis.bonheure@ulb.ac.be}
\email{hussein.cheikh-ali@ulb.ac.be}
\email{robson.nascimento@ulb.ac.be}

\date{}

\thanks{The authors are gratefully indebted to Bruno Premoselli for many discussions and helpful criticism. The authors were supported by PDR T.1110.14F (FNRS). D. Bonheure was partially supported by the project ERC Advanced Grant  2013 n. 339958: ``Complex Patterns for Strongly Interacting Dynamical Systems - COMPAT''}
\begin{abstract}
We consider the best constant in a critical Sobolev inequality of second order. We show non-rigidity for the optimizers above a certain threshold, namely we prove that the best constant is achieved by a non-constant solution of the associated fourth-order elliptic problem under Neumann boundary conditions. Our arguments rely on asymptotic estimates of the Rayleigh quotient. We also show rigidity below another threshold.  
\end{abstract}

\maketitle

\section{Introduction}
It is well-known that the Sobolev inequality 
\begin{equation}\label{Sobolev}
S\|u\|_{L^{2_{\ast}}(\R^N)}\le \|Du\|_{L^2(\R^N)}, 
\end{equation}
where $2_{\ast}=2N/(N-2)$, $N\ge 3$ and $S=S(N)$ is a positive constant, plays a fundamental role in geometric analysis. A simple scaling argument shows that the exponent $2_{\ast}$ is the only possible one in the inequality. This very same scaling argument implies that the embedding of $H^1_0(\Omega)$ into $L^{2_{\ast}}(\Omega)$, where $\Omega$ is a bounded open set, cannot be compact. This lack of compactness is the genesis of one of the main complexity in the celebrated Yamabe problem \cite{MR0125546}. 
Assume that $N\ge 3$ and $(M,g)$ is a compact Riemannian $N$-dimensional manifold with scalar curvature $R_g$. The Yamabe Problem consists in looking for a metric $g'$ conformally equivalent to $g$ such that the scalar curvature $R_{g'} \equiv 1$. It happens that this problem amounts to finding a positive solution of
$$-4\frac{N-1}{N-2}\Delta_g u + R_g u =|u|^{\frac{4}{N-2}}u,$$
where $\Delta_g$ denotes the Laplace-Beltrami operator on $M$. Then $g^{\prime} = u^{\frac{4}{N-2}}g$ is a conformal metric satisfying $R_{g^{\prime}}\equiv 1$. The first approach by Yamabe was corrected by Trudinger \cite{MR0240748}. 
We refer to \cite{MR888880,MR1636569} for the history of the problem and to Aubin \cite{Aubin76,MR0448404}, Schoen \cite{MR788292} and Schoen and Yau \cite{MR931204} for the main breakthroughs in its resolution. 

Assume now that $\Omega\subset\R^N$ $(N\geq 3)$ is an open bounded domain with a smooth boundary, and $\alpha$ is a positive real number. In their seminal paper, Brezis and Nirenberg \cite{BreNi83} have proved that the existence of a (positive) solution to the problem
\begin{equation}\label{model-problem}
-\Delta u+\alpha u=|u|^{\frac{4}{N-2}}u\hspace{.2cm}\text{in}\hspace{.2cm}\Omega,
\end{equation}
under homogeneous Dirichlet boundary condition, is closely related to the best constant for the Sobolev embedding of $H^{1}_{0}(\Omega)$ into $L^{\frac{2N}{N-2}}(\Omega)$. Their arguments are inspired by the work of Aubin \cite{Aubin76} on the Yamabe problem. The idea consists in minimizing the Rayleigh quotient 
\begin{equation*}
Q_{\alpha}(u)=\frac{\int_{\Omega}(|\nabla u|^2+\alpha|u|^2)\, dx}{(\int_{\Omega}|u|^{\frac{2N}{N-2}}\, dx)^{\frac{N-2}{N}}},\ u\in H^{1}_{0}(\Omega)\setminus\{0\},
\end{equation*}
and in evaluating $Q_{\alpha}$ at test functions of the form $u(x)=\varphi(x)(\varepsilon+|x|^2)^{-\frac{N-2}{2}}$, where $\varphi$ is a cut-off function. The functions $(\varepsilon+|x|^2)^{-\frac{N-2}{2}}$ play a natural role because they are extremal functions for the Sobolev inequality \eqref{Sobolev}, see, for instance \cite{Aubin76,Talenti}.

Brezis \cite[Section 6.4]{Bre87} suggested to study \eqref{model-problem} under Neumann boundary condition. It happens that the equation is then related to models in mathematical biology such as the Keller-Segel model \cite{KelSeg,LinNiTakagi88,MR849484} for chemotaxis and the shadow system of Gierer and Meinhardt \cite{Gierer1972,MR2103689}. 

With Neumann boundary conditions, the equation \eqref{model-problem} admits the constant solutions $u\equiv 0$ and $u\equiv\alpha^{(N-2)/4}$. When the nonlinearity in \eqref{model-problem} is subcritical (namely when the exponent $4/(N-2)$ is replaced by $q-2$ with $2<q<2_{\ast}$), Lin, Ni and Tagaki \cite{LinNiTakagi88} have proved that the only positive solution to \eqref{model-problem}, for small $\alpha>0$, is the nonzero constant solution. As a byproduct, this yields directly the sharp constant $C(\alpha)= \alpha^{\frac12} |\Omega|^{\frac12-\frac1q}$ in the inequality
$$
\left(\int_{\Omega}(|\nabla u|^2+\alpha|u|^2)\, dx\right)^{1/2}\ge C(\alpha)\left(\int_{\Omega}|u|^{q}\, dx\right)^{\frac{1}{q}},
$$
for $u\in H^{1}(\Omega)$. In the critical case, Lin and Ni \cite{LinNi86} raised this rigidity result as a conjecture. 

\medbreak
\noindent\textsc{Lin-Ni's conjecture:}
\emph{For $\alpha$ small enough, Equation \eqref{model-problem} under Neumann boundary condition admits only $\alpha^{(N-2)/4}$ as a positive solution.}

\medbreak

In the subcritical case, it is easily seen from a Morse index argument that the rigidity is broken for large $\alpha$. 
In the critical case, inspired by Brezis and Nirenberg, Wang \cite[Theorem 3.1]{Wang91}, and Adimurthi and Mancini \cite[Theorem 1.1]{AdiMan91} proved that Equation \eqref{model-problem} under Neumann boundary conditions admits a non-constant (least energy) positive solution $u(\alpha)$ for every $\alpha>\overline{\alpha}>0$. These least energy solutions $u(\alpha)$ have the following concentration property \cite{AdiPaceYada92,MR1174600}: they are single-peaked in the sense that every $u(\alpha)$, for $\alpha>0$ sufficiently large, attains its unique maximum at a point $p(\alpha)\in\partial\Omega$, and $p(\alpha)\to p_{0}\in\partial\Omega$ as $\alpha\to\infty$, with $H(p_0)=\max_{p\in\partial\Omega}H(p)$, where $H(p)$ is the mean curvature of $\Omega$ at $p\in\partial\Omega$. Such concentration behaviour was shown in the subcritical case by Ni and Takagi \cite{MR1115095,MR1219814}. 

In the last three decades, lots of progress has been made towards proving or disproving Lin-Ni's conjecture. It is a difficult task to exhaust all the related literature concerning this conjecture and it is not our purpose here. Nevertheless, we give a short overview of the main results regarding this conjecture.  In case $\Omega=B_{R}(0)$, and $u$ is a radial function, the conjecture was studied by Adimurthi and Yadava \cite{AdimYad91}, and Budd, Knaap and Peletier \cite{BuddKnaPel91}. Namely, they investigated the problem
\begin{equation}\label{P-Lin-Ni}
\begin{cases}
-\Delta u+\alpha u=u^{\frac{N+2}{N-2}}&\text{in}\ B_{R}(0)\\
u\ \text{is radial}\hspace{.1cm}\text{and}\hspace{.1cm}u>0&\text{in}\ B_{R}(0)\\
\partial_{r}u=0&\text{on}\ \partial B_{R}(0),
\end{cases}
\end{equation}
where $\partial_{r}u:=\frac{x}{|x|}\cdot\nabla u$. They have established the following result.
\begin{thma}
For $\alpha>0$ sufficiently small, the following statements hold:
\begin{enumerate}[(a)]
\item If $N=3$ or $N\geq 7$, then \eqref{P-Lin-Ni} admits only the constant solution.
\item If $N\in\{4,5,6\}$, then \eqref{P-Lin-Ni} admits a nonconstant solution. 
\end{enumerate}
\end{thma}

Theorem A highlights that the validity of Lin-Ni's conjecture depends on the dimension. The proof of Theorem $A$ uses radial symmetry to reduce \eqref{P-Lin-Ni} to the ODE
$$-u'' - \frac{N-1}{r}u'+ \alpha u = u^{\frac{N+2}{N-2}}\quad \text{in }(0,R),$$
with the boundary condition $u'(R)=0$ (the second boundary condition $u'(0)=0$ comes from the assumption of radial symmetry of the solution).  With regard to Lin-Ni's conjecture in general domains, such an approach cannot be applied. When $\Omega$ is a convex domain, Zhu \cite{Zhu99} proved that the conjecture is true if $N=3$, see also \cite{WeiXu05,MR2410880}. In case $\Omega$ is a smooth bounded domain, and the mean curvature of $\Omega$ is positive along the boundary $\partial\Omega$, Druet, Robert and Wei \cite{DruetRobWei12} have proved that Lin-Ni's conjecture is true for $N=3$ and $N\geq7$, assuming a bound on the energy of solutions. If $\Omega$ is any smooth and bounded domain, Rey and Wei \cite{ReyWei05} have proved that the conjecture is false if $N=5$. The Lin-Ni conjecture is wrong in all dimensions in non-convex domains \cite{MR2645043} and in dimension $N\ge 4$ for convex domains \cite{MR2806101}.
If we restrict our attention to least energy solutions, Adimurthi and Yadava have proved that the conjecture holds in every dimension \cite{AdimYad93}.

Motivated by the above results, our purpose in this paper is to study an analogue of Equation \eqref{model-problem} involving a fourth-order elliptic operator. Namely, we assume that $\Omega\subset\R^N$ ($N\geq 5$) is an open bounded set with smooth boundary, and $\alpha$ is a positive parameter. We are interested in the following problem
\begin{align}
\tag{\protect{$P_{\nu}$}}\label{problem-P-alpha}
\begin{cases}
\Delta^2 u -\Delta u+\alpha u=|u|^{\frac{8}{N-4}}u,&\mbox{in}\ \Omega,\\
\partial_{\nu} u =\partial_{\nu}(\Delta u)= 0,&\mbox{on}\ \partial\Omega.
\end{cases}
\end{align}
The linear operator $\Delta^2 -\Delta+\alpha$ is often referred to as a  Paneitz-Branson type operator \cite{DjaHeLe00} with constant coefficients. If $(M, g)$ is a compact Riemannian manifold of dimension $N\ge 5$ and $Q_g$ is its $Q$ curvature \cite{MR832360}, the prescribed $Q$ curvature problem consists in finding metric of constant $Q$ curvature in the conformal class of $g$, see for instance \cite{Chang-Yang,Dja-Mal,Li-Xiong}. 
This amounts to finding a positive solution to
\begin{equation}
\label{eqin}
P_g(u) = |u|^{\frac{8}{N-4}},
\end{equation}
where $P_g$ is the Paneitz operator \cite{MR2393291}, i.e., 
$$
P_g u:= \Delta_g^2 u-div_g (A_g du) +h u,
$$
with
\begin{equation}
\label{paneitz}
A_g= \dfrac{(N-2)^2+4}{2(N-1)(N-2)}R_g g -\dfrac{4}{N-2}Ric_g,
\end{equation}
where $R_g$ (resp. $Ric_g$) stands for the scalar curvature (resp. Ricci curvature), and $h=\dfrac{N-4}{2}Q_g$ where $Q_g$ is the $Q$ curvature which is defined by
$$Q_g=\dfrac{1}{2(N-1)}\Delta_g R_g+\dfrac{N^3 -4N^2+16N - 16}{8(N-1)^2(N-2)^2}R_g^2-\dfrac{2}{(N-2)^2}|Ric_g|_g^2.$$
Equation \eqref{eqin} is referred to as the Paneitz-Branson equation.  In addition to the above mentioned contributions, we refer to \cite{MR2819586,BAKRI2014118,Bakri_2015,MR2083484,MR3669775,MR3085755,HangYang} and the references therein for a glance to recent results. If $(M,g)$ is Einstein ($Ric_g =\lambda g$, $\lambda\in \R$), then the Paneitz-Branson operator takes the form
\begin{equation}
\label{paneitzcons}
P_g u =\Delta^2_g u +b\Delta_g u +cu,
\end{equation}
where $b=\dfrac{N^2-2N-4}{2(N-1)}\lambda$ and $c=\dfrac{N(N-4)(N^2-4)}{16(N-1)^2}\lambda^2$, see \cite{MR2819586}. 
Observe that in the geometrical context, $b$ and $c$ in \eqref{paneitzcons} can have the same sign while we take $b=-1$ and $\alpha>0$ in \eqref{problem-P-alpha}. The case of a positive Laplacian interacting with the bi-Laplacian will be considered in a future work. 

Again, Equation \eqref{problem-P-alpha} is critical since the $L^{\frac{2N}{N-4}}$-norm scales like the $L^2$-norm of the Laplacian. The problem admits two constant solutions, namely, $u_0=0$, and $u_{1}=\alpha^{\frac{N-4}{8}}$. When the power nonlinearity is subcritical, one can prove in a standard way (adapting for instance \cite{MR849484}) that any positive solution is constant (and nonzero) when $\alpha$ is small whereas this rigidity breaks down for large $\alpha$. Our main concern is to establish a non rigidity result for \eqref{problem-P-alpha} when $\alpha$ is large. To this end, we establish some Sobolev inequalities of second order with respect to the functional space associated to variational solutions to \eqref{problem-P-alpha}. To the best of our knowledge, there are not many works in the literature dealing with the boundary conditions of \eqref{problem-P-alpha}. For a general overview of this subject we refer to the work of Berchio and Gazzola \cite{BerGaz06}, where the problem of embeddings of second order Sobolev spaces with traces on the boundary has been studied. For more insight on polyharmonic operators, we refer to \cite{GaGruSwe10}.

Our main result is stated as follows. Let 
$
H^{2}_{\nu}(\Omega)=\{u\in H^2(\Omega):\partial_{\nu} u=0\ \text{on}\ \partial\Omega\},
$
and 
\begin{equation}\label{inf-best}
 \Sigma_{\nu}(\Omega):=\inf_{u\in M_{\Omega}}J(u),
\end{equation}
where
\begin{align*}
J(u)&=\int_{\Omega}(|\Delta u|^2+|\nabla u|^2+\alpha |u|^2)\, dx,\\
\intertext{and}
 M_{\Omega}&=\left\{u\in H^{2}_{\nu}(\Omega):\int_{\Omega}|u|^{\frac{2N}{N-4}}\, dx=1\right\}.
\end{align*}
A solution of \eqref{problem-P-alpha} is said to be of least energy if its $L^{\frac{2N}{N-4}}$ normalized multiple is an optimizer for \eqref{inf-best}.

\begin{theorem}\label{main-theo-I}
Assume $\Omega$ is an open bounded subset of $\R^N$ with smooth boundary. There exists $\overline{\alpha}=\overline{\alpha}(N,|\Omega|)>0$ such that for $\alpha>\overline{\alpha}$, any least energy solution of Equation \eqref{problem-P-alpha} is nonconstant. 
\end{theorem}

\medbreak

It is worth mentioning that when \eqref{problem-P-alpha} is considered in a smooth compact Riemannian manifold, Felli, Hebey and Robert \cite{FeHeFre05} have established that for any $\Lambda>0$, there exists $\alpha_0>0$ such that for $\alpha\geq\alpha_0$, the above equation does not have a solution whose energy is smaller than $\Lambda$.

Since \eqref{problem-P-alpha} is critical, the existence of a nontrivial solution does not follow directly from standard variational methods. Moreover, the functional setting brings new difficulties in comparison to the second order conterpart \eqref{model-problem}. To overcome the lack of compactness, we follow the arguments introduced in Aubin \cite{Aubin76}, and developed in Brezis and Nirenberg \cite{BreNi83}. However, due to the boundary conditions of \eqref{problem-P-alpha}, we cannot apply the arguments from Adimurthi and Mancini \cite{AdiMan91}, Wang \cite[Theorem 3.1]{Wang91} nor those from Berchio and Gazzola \cite{BerGaz06}. As a way out, our approach consists in making a change of coordinates in such a way that part of the boundary $\partial\Omega$ will be diffeomorphic to a flat subset of $\R^N$. Roughly speaking, the idea is to straighten out the boundary and then to estimate the Rayleigh quotient by choosing suitable test functions adapted to these new coordinates. In this new coordinate system, we establish the following second order Sobolev inequality. The Sobolev constant $S$ is defined from now on by 
\begin{equation*}
S=\inf_{u\in\mathcal{D}^{2,2}(\R^N)}\left\{\int_{\R^N}|\Delta u|^2\, dx:\int_{\R^N}|u|^{\frac{2N}{N-4}}\, dx=1\right\}.
\end{equation*}

\begin{lemma}\label{sharp-H2-mu}
Assume that $\Omega$ is an open bounded subset of $\R^N$ with smooth boundary and $N\geq 5$. Then, for every $\varepsilon>0$, there exists $B(\varepsilon)>0$ such that for all $u\in H^{2}_{\nu}(\Omega)$,
\begin{equation}\label{DHL-Lemma}
\|u\|^{2}_{L^{\frac{2N}{N-4}}(\Omega)}\leq\left(\frac{2^{4/N}}{S}+\varepsilon\right)\|\Delta u\|^{2}_{L^{2}(\Omega)}+B(\varepsilon)\|u\|^{2}_{H^1(\Omega)}.
\end{equation}
Moreover, $\Sigma_{\nu}(\R^{N}_{+})=S/2^{4/N}$ and the infimum is not achieved.
\end{lemma}
This lemma is the key to prove Theorem \ref{main-theo-I}. We believe this Sobolev inequality has its own interest and can be useful in other situations.
As in the second order case, if we focus on least energy solutions, then \eqref{problem-P-alpha} has only the constant solution $\alpha^{\frac{N-4}{8}}$ when $\alpha$ is small enough. 
\begin{theorem}\label{rigidity}
Assume $\Omega$ is an open bounded subset of $\R^N$ with smooth boundary. Then, there exists $\underline{\alpha}=\underline{\alpha}(N,|\Omega|)>0$ such that for $0<\alpha<\underline{\alpha}$, the only least energy solution of Equation \eqref{problem-P-alpha} is the constant solution $\alpha^{\frac{N-4}{8}}$.
\end{theorem}

A natural question arises from Theorem \ref{main-theo-I} and Theorem \ref{rigidity}: where does rigidity of the minimizer break down? A tempting conjecture is that the rigidity is lost when the constant solution loses its stability.  However this is still open even for the second order equation \eqref{P-Lin-Ni}, see for instance \cite{MR3564729,Dolbeault_2017}.
 
The manuscript is organized as follows. In Section \ref{Prelim-4}, we settle the functional setting and recall some known facts regarding best constants for embeddings of second order Sobolev spaces. In Section \ref{sec-const-4}, we establish a relation between the best constant for the second order Sobolev embedding and that of the functional space associated to \eqref{problem-P-alpha}. In Section \ref{sec-estimates-4}, by taking into account the smoothness of the boundary $\partial\Omega$ and the effect of the principal curvatures, we establish some asymptotic estimates, and  we give the proof of Theorem \ref{main-theo-I}. In Section \ref{Section-2nd-Sob-ineq}, we establish some Sobolev inequalities of second order. Section \ref{minimizers} contains the proof of the rigidity theorem for small $\alpha$. 
In forthcoming works, we will consider the counterpart of Lin-Ni's conjecture for small $\alpha$ and study the critical dimensions.

\section{Preliminaries}\label{Prelim-4}
In this section we settle the functional setting regarding \eqref{problem-P-alpha}, and recall some known facts about the best constants of some second-order Sobolev embeddings.

A classical result in the theory of Sobolev spaces claims that if $\Omega\not\equiv\R^N$ is a smooth domain, then any function in $H^2(\Omega)$ admits some traces on the boundary $\partial\Omega$, see, for instance, \cite[Theorem 7.53]{Adams75}, or \cite[Lemmas 16.1 \& 16.2]{Tartar07}. In particular, there exists a linear continuous operator
\begin{equation*}
\Tr:H^2(\Omega)\to H^{3/2}(\partial\Omega)
\end{equation*}
such that $\Tr u=\partial_{\nu} u\big{|}_{\partial\Omega}$ for all $u\in C^1(\overline{\Omega})$. In the sense of traces, the kernel of the operator $\Tr$ gives rise to the following proper subspace of $H^2(\Omega)$,
\begin{equation*}
H^{2}_{\nu}(\Omega):=\{u\in H^2(\Omega):\partial_{\nu} u=0\ \text{on}\ \partial\Omega\}.
\end{equation*}
We recall that $H^2(\Omega)$ is a Hilbert space endowed with the inner product defined through
\begin{equation*}
\langle u,v \rangle=\int_{\Omega}(D^2 u D^2 v+Du Dv+uv)\, dx\hspace{.3cm}\text{for all}\ u,v\in H^2(\Omega).
\end{equation*}
Using regularity theory, see, for instance, \cite[Theorem 8.12]{GilTru98}, or \cite[Chapter 1, Section 6, Theorem 4]{Krylov08}, we infer that 
\begin{equation*}
(u,u)\mapsto \|u\|_{H^{2}_{\nu}(\Omega)}:=\left(\int_{\Omega}(|\Delta u|^2+|\nabla u|^2 +\alpha |u|^2)\, dx\right)^{1/2}
\end{equation*}
defines an equivalent norm in $H^{2}_{\nu}(\Omega)$ when $\alpha>0$.

Note that by integration by parts, $H^{2}_{\nu}(\Omega)$ is the natural space for (weak) solutions to \eqref{problem-P-alpha}. To obtain nontrivial least energy solutions to \eqref{problem-P-alpha}, we  consider the minimization problem \eqref{inf-best}.

Before proceeding any further, we establish some notations and recall some known results. Denote by $\mathcal{D}^{2,2}(\R^N)$ the closure of the space of smooth compactly supported functions in $\R^N$ with respect to the norm $\|D^{2}\cdot\|_{L^2(\R^N)}$. Note that integration by parts two times together with a density argument show that $\|D^2 \phi\|_{L^{2}(\R^N)}=\|\Delta \phi\|_{L^{2}(\R^N)}$ for all $\phi\in\mathcal{D}^{2,2}(\R^N)$. It is well-known that the best constant for the embedding of $\mathcal{D}^{2,2}(\R^N)$ into $L^{\frac{2N}{N-4}}(\R^N)$ might be characterized by
\begin{equation}\label{S-inf-def}
S:=\inf_{u\in\mathcal{D}^{2,2}(\R^N)}\left\{\int_{\R^N}|\Delta u|^2\, dx:\int_{\R^N}|u|^{\frac{2N}{N-4}}\, dx=1\right\}.
\end{equation}
We recall that Lieb \cite[Section IV]{Lieb83}, and Lions \cite[Theorem I.1]{Lions85} (see also \cite{MR705681}) have proved that there exists a minimizer for \eqref{S-inf-def}, which is uniquely determined up to translations and dilations. Namely, the minimizer is given by the one-parameter family 
\begin{equation}\label{U-epsi}
 u_{\varepsilon}(x):=\gamma_N\frac{\varepsilon^{\frac{N-4}{2}}}{(\varepsilon^2+|x|^2)^{\frac{N-4}{2}}},
\end{equation}
where
\begin{equation}\label{Gamma-def-opt}
 \gamma_N:=\left[(N-4)(N-2)N(N+2)\right]^{\frac{N-4}{8}}.
\end{equation}
With the above expression, the constant $S$ can be evaluated explicitly
\begin{equation*}
S=\pi^{2}(N-4)(N-2)N(N+2)\left(\frac{\Gamma(\frac{N}{2})}{\Gamma(N)}\right)^{\frac{4}{N}}.
\end{equation*}
Note that $u_{\varepsilon}(x)=\varepsilon^{-\frac{N-4}{2}}u_1(\frac{x}{\varepsilon})$, and $u_{\varepsilon}$ satisfies the equation
\begin{equation}\label{u-satisfy-equal}
\Delta^2 u_{\varepsilon}=u_{\varepsilon}|u_{\varepsilon}|^{\frac{8}{N-4}}\hspace{.3cm}\text{in}\hspace{.3cm}\R^N.
\end{equation}
In fact, all positive solutions of the above equation are given by the ${\varepsilon}$- family \eqref{U-epsi}. In regard to this result, see, for instance, \cite[Theorem 2.1]{EdForJan90}, \cite[Theorem 1.3]{Lin98}, and \cite[Theorem 1.3]{WeiXu99}. 
 
Now we recall that for any smooth bounded domain $\Omega\subset\R^N$, $H^{2}_{0}(\Omega)$ and $H^2\cap H^{1}_{0}(\Omega)$ are Hilbert spaces endowed with the equivalent norm defined by
\begin{equation}\label{Delta-norm}
u\mapsto\left(\int_{\Omega}|\Delta u|^2\, dx\right)^{1/2},
\end{equation}
see, for instance, \cite[Theorem 2.31]{GaGruSwe10}. Here we note that $H^2\cap H^{1}_{0}(\Omega)$ is the space where variational solutions to fourth-order elliptic PDEs are sought when complemented with the so-called homogenous Navier boundary conditions along the boundary, $u=\Delta u=0$ on $\partial\Omega$, while $H^{2}_{0}(\Omega)$ is the functional space for variational solutions when Dirichlet boundary conditions are considered, $u=\partial_{\nu}u=0$ on $\partial\Omega$. Observe that $H^2\cap H^{1}_{0}(\Omega)$ strictly contains $H^{2}_{0}(\Omega)$.

The question whether or not the best constants for the embeddings of $H^{2}_{0}(\Omega)$ and $H^2\cap H^{1}_{0}(\Omega)$ into $L^{\frac{2N}{N-4}}(\Omega)$ are equal, and independent of the domain, was investigated by van der Vorst \cite[Theorems 1 and 2]{Vorst93}. He has shown that for any smooth domain $\Omega\subset\R^N$,
\begin{align*}
S&=\inf_{u\in H^{2}_{0}(\Omega)}\left\{\int_{\Omega}|\Delta u|^2\, dx:\int_{\Omega}|u|^{\frac{2N}{N-4}}\, dx=1\right\}\\
&=\inf_{u\in H^{2}\cap H^{1}_{0}(\Omega)}\left\{\int_{\Omega}|\Delta u|^2\, dx:\int_{\Omega}|u|^{\frac{2N}{N-4}}\, dx=1\right\}
\end{align*}
and that the infimum is never achieved when $\Omega$ is bounded. However, a crucial part of the proof is not carried out in full detail. In addition, it is not clear that \cite[Lemma A1]{Vorst93} can be proved using an extension argument. In regard to this result, we refer to \cite[Theorem 1]{GaGruSwe-2010}.

In contrast with the above results, we cannot expect to obtain the same conclusions with respect to the space $H^{2}_{\nu}(\Omega)$ since \eqref{Delta-norm} is no longer a norm in $H^{2}_{\nu}(\Omega)$. 

\section{A relation between $\Sigma_{\nu}(\R^N)$ and $S$}\label{sec-const-4}
In this section we show that $\Sigma_{\nu}(\R^N)$ and $S$ are equal. For convenience, throughout the rest of this paper we denote
\begin{equation*}
2^{*}=\frac{2N}{N-4}.
\end{equation*}
Recall that as a consequence of the density of the space of smooth compactly supported functions in $\R^N$ with respect to the $H^2$-Sobolev norm,
\begin{equation*}
H^{2}_{0}(\R^N)=H^2\cap H^{1}_{0}(\R^N)=H^{2}_{\nu}(\R^N)=H^{2}(\R^N).
\end{equation*}

Our next result is inspired by \cite[Theorem 1(i)]{BerGaz06}.
\begin{lemma}\label{Sigma=S-Rn}
Assume $N\geq 5$ and let $S$ be defined as in \eqref{S-inf-def}. Then, for any $\alpha>0$,
\begin{equation*}
\Sigma_{\nu}(\R^N)=S,\hspace{.3cm}\text{and}\hspace{.3cm}\text{the infimum is never achieved}.
\end{equation*}
\end{lemma}

\begin{proof}
We begin by noticing that
\begin{align}\label{S-leq-S_alpha}
S&=\inf_{u\in\mathcal{D}^{2,2}(\R^N)}\left\{\int_{\R^N}|\Delta u |^{2}\, dx:\int_{\R^N}|u|^{2^{*}}\, dx=1\right\}\nonumber\\[.5em]
&\leq\inf_{u\in H^2(\R^N)}\left\{\int_{\R^N}|\Delta u |^{2}\, dx:\int_{\R^N}|u|^{2^{*}}\, dx=1\right\}\nonumber\\[.5em]
&\leq\inf_{u\in H^2(\R^N)}\left\{\int_{\R^N}(|\Delta u |^{2}+|\nabla u|^2+\alpha |u|^2)\, dx:\int_{\R^N}|u|^{2^{*}}\, dx=1\right\}&\nonumber\\[.5em]
&=\Sigma_{\nu}(\R^N), 
\end{align}
where in the first inequality we have used the fact that $H^2(\R^N)\subset\mathcal{D}^{2,2}(\R^N)$. Now, in order to show the reverse inequality in \eqref{S-leq-S_alpha} we proceed as follows.

\indent\emph{Step one: For all $N\geq 5$, there holds $\Sigma_{\nu}(\R^N)\leq S$.} 
We construct a suitable minimizing sequence for which $\Sigma_{\nu}(\R^N)\leq S$. For convenience we write $|x|=r$. For all $\varepsilon>0$, we consider the function
 \begin{align*}
 \vartheta_{\varepsilon}(r)&:=u_{\varepsilon}(r)-u_{\varepsilon}(1)\\
 &=\gamma_{N}\varepsilon^{\frac{N-4}{2}}\left(\frac{1}{(\varepsilon^2+r^2)^{\frac{N-4}{2}}}-\frac{1}{(\varepsilon^2+1)^{\frac{N-4}{2}}}\right).
 \end{align*}
Now we set
 \begin{align}\label{z-epsi-sharp-half}
z_{\varepsilon}(r)=\begin{cases}
   \vartheta_{\varepsilon}(r),&\text{if}\ 0< r\leq 1/2\\
    w_{\varepsilon}(r),&\text{if}\ 1/2\leq r \leq 1\\
    0, &\text{if}\ r\geq 1,
   \end{cases}
\end{align}
 where $w_{\varepsilon}(r):=a(\varepsilon)(r-1)^3+b(\varepsilon)(r-1)^2$, with $a(\varepsilon)$, and $b(\varepsilon)$ chosen in such a way that for $r_0=1/2$, 
 \begin{equation*}
w_{\varepsilon}(r_0)=\vartheta_{\varepsilon}(r_0),\hspace{.3cm}\text{and}\hspace{.3cm}\partial_{r}w_{\varepsilon}(r_0)=\partial_{r}\vartheta_{\varepsilon}(r_0).
 \end{equation*}
 In particular, 
 \begin{equation}\label{a-b-epsilon}
 a(\varepsilon)=O(\varepsilon^{\frac{N-4}{2}}),\hspace{.3cm}\text{and}\hspace{.3cm}b(\varepsilon)=O(\varepsilon^{\frac{N-4}{2}}).
 \end{equation}
In this way, for every $N\geq 5$, since $z_{\varepsilon}$ is a $C^1$ gluing, $w_{\varepsilon}(1)=0$, and $\partial_{r}w_{\varepsilon}(1)=0$, we infer that $z_{\varepsilon}\in H^2(\R^N)$. 

Next, we seek an upper bound for the functional $J$ evaluated at $z_{\varepsilon}/ \|z_{\varepsilon}\|_{L^{2^{*}}(\R^N)}$. Indeed, arguing as in \cite[(7.58)]{GaGruSwe10}, 
\begin{equation}\label{laplacian-o}
 \int_{|x|\leq 1/2}|\Delta u_{\varepsilon}|^2\, dx=S^{N/4}+O(\varepsilon^{N-4}).
 \end{equation}
From this together with \eqref{a-b-epsilon},
\begin{align}
 \int_{\R^N}|\Delta z_{\varepsilon}(|x|)|^2\, dx &=\int_{|x|\leq 1/2}|\Delta \vartheta_{\varepsilon}(|x|)|^2\, dx + \int_{1/2\leq |x|\leq 1}|\Delta w_{\varepsilon}(|x|)|^2\, dx\nonumber\\[.4em]
 &=\int_{|x|\leq 1/2}|\Delta u_{\varepsilon}(|x|)|^2\, dx +o(1)\nonumber\\[.4em]
 &=S^{N/4}+o(1).\label{estimate-lapla-z}
 \end{align}
Similarly, by \eqref{a-b-epsilon},
\begin{equation}\label{L2-critical-L2}
\int_{\R^N}|z_{\varepsilon}(|x|)|^{2^{*}}\, dx=S^{N/4}+o(1),\hspace{.3cm}\text{and}\hspace{.3cm}\int_{\R^N}|z_{\varepsilon}(|x|)|^{2}\, dx=o(1).
\end{equation}
Since $z_{\varepsilon}\in H^2(\R^N)$ we use interpolation between $\|\Delta z_{\varepsilon}\|_{L^2(\R^N)}$ and $\|z_{\varepsilon}\|_{L^2(\R^N)}$ to get an estimate for the $L^2$-norm of $\nabla z_{\varepsilon}$. Namely, by \eqref{estimate-lapla-z} and \eqref{L2-critical-L2},
\begin{equation}\label{estimate-grad-z}
\int_{\R^N}|\nabla z_{\varepsilon}(|x|)|^2\, dx=o(1).
\end{equation}
Finally, we set $Z_{\varepsilon}:=z_{\varepsilon}/\|z_{\varepsilon}\|_{L^{2^{*}}(\R^N)}$ so that $Z_{\varepsilon}\in M_{\R^N}$. Therefore, in view of \eqref{estimate-lapla-z}-\eqref{estimate-grad-z},
\begin{align*}
\Sigma_{\nu}(\R^N)&\leq J(Z_{\varepsilon}),\hspace{.3cm}\text{for all}\ \varepsilon>0,\\
&=\frac{S^{N/4}+o(1)}{(S^{N/4}+o(1))^{\frac{N-4}{N}}}\hspace{.3cm}\text{as}\hspace{.3cm}\varepsilon\to 0. 
\end{align*}
This proves that $\Sigma_{\nu}(\R^N)\leq S$, and hence the first part of the lemma follows.

\indent\emph{Step two: $\Sigma_{\nu}(\R^N)$ is never achieved.} Seeking a contradiction, we assume that there exists a function $u\in M_{\R^N}$ which achieves equality in \eqref{inf-best}. Define for $\lambda>0$ the rescaled function $u_{\lambda}(x):=\lambda^{\frac{N-4}{2}}u(\lambda x)$ so that $\|u_{\lambda}\|^{2^{*}}_{L^{2^{*}}(\R^N)}=1$. Thus,
\begin{align*}
\int_{\R^N}(|\Delta u|^2+|\nabla u|^2+\alpha |u|^2)\, dx&\leq\int_{\R^N}(|\Delta u_{\lambda}|^2+|\nabla u_{\lambda}|^2+\alpha|u_{\lambda}|^2)\, dx\\
&\leq\int_{\R^N}|\Delta u|^2\, dx+\frac{1}{\lambda^2}\int_{\R^N}|\nabla u|^2\, dx+\frac{\alpha}{\lambda^4}\int_{\R^N}|u|^2\, dx.
\end{align*}
By sending $\lambda$ to infinity in the above inequality we obtain a contradiction.
\end{proof}

Now we recall that by the Sobolev embedding theorem, there exist positive constants $A$ and $B$ such that for any $u\in H^2(\Omega)$,
\begin{equation*}
\|u\|^{2}_{L^{2^{*}}(\Omega)}\leq A \|\Delta u\|^{2}_{L^2(\Omega)}+B\|u\|^{2}_{H^1(\Omega)}.
\end{equation*}
The task of finding the best constants in the above inequality has been extensively studied in the last years. In this regard, we refer to \cite{DjaHeLe00,He03}. In this direction, we will prove that 
for every $\varepsilon>0$, there exists $B(\varepsilon)>0$ such that for all $u\in H^{2}_{\nu}(\Omega)$,
\begin{equation}\label{DHL-Lemma-bis}
\|u\|^{2}_{L^{2^{*}}(\Omega)}\leq\left(\frac{2^{4/N}}{S}+\varepsilon\right)\|\Delta u\|^{2}_{L^{2}(\Omega)}+B(\varepsilon)\|u\|^{2}_{H^1(\Omega)}.
\end{equation}
Moreover, $\Sigma_{\nu}(\R^{N}_{+})=S/2^{4/N}$, and the infimum is not achieved.
This is the content of Lemma \ref{sharp-H2-mu} proved in Section \ref{Section-2nd-Sob-ineq}. 
As a consequence of inequality \eqref{DHL-Lemma-bis}, we establish the following result.


\begin{lemma}\label{S_alphaK}
Assume that $\Omega$ is an open bounded subset of $\R^N$ with smooth boundary and $N\geq 5$. If $\Sigma_{\nu}(\Omega)<S/{2^{4/N}}$, then the infimum in \eqref{inf-best} is achieved. 
\end{lemma}
\begin{proof}
Let $(u_k)_{k\in\N}\subset M_{\Omega}$ be a minimizing sequence for $\Sigma_{\nu}(\Omega)$. Since $J$ is the square of a norm in $H^{2}_{\nu}(\Omega)$ we deduce that the sequence $(u_k)_{k\in\N}$ is bounded in $H^{2}_{\nu}(\Omega)$. Consequently, up to extracting a subsequence, there exists $u\in H^{2}_{\nu}(\Omega)$ with
 \begin{align}\label{Sobolev-propertiesI}
   \begin{cases}
     u_k\rightharpoonup u & \text{weakly in}\ H^{2}_{\nu}(\Omega)\\
     u_k\rightharpoonup u & \text{weakly in}\ L^{2^{*}}(\Omega)\\
     u_k\to u & \text{strongly in}\ H^1(\Omega)\\
     u_k(x)\to u(x) & \text{a.e. in}\ \Omega.
   \end{cases}
\end{align}
 
\indent\emph{Step one: There holds $u\not\equiv 0$.} Seeking a contradiction, we assume that $u\equiv 0$. By \eqref{Sobolev-propertiesI},
 \begin{equation}\label{H1-convergence}
  u_k\to 0\hspace{.3cm}\text{strongly in}\ H^1(\Omega).
 \end{equation}
Recall that $\|u_k\|_{L^{2^{*}}(\Omega)}=1$. Thus,  
\begin{align*}
 \lim_{k\to\infty}\int_{\Omega}|\Delta u_k|^2\, dx&\leq \lim_{k\to\infty}J(u_k)\\
 &= \Sigma_{\nu}(\Omega)\\
 &\leq \Sigma_{\nu}(\Omega)\left(\frac{2^{4/N}}{S}+\varepsilon\right)\int_{\Omega}|\Delta u_k|^2\, dx + o(1)\hspace{.5cm}(\ \text{by \eqref{DHL-Lemma}} )
\end{align*}
for every $\varepsilon>0$. Note that $\Sigma_{\nu}(\Omega)>0$. Hence, as a consequence of the above inequality,
\begin{equation*}
 1\leq \Sigma_{\nu}(\Omega)\left(\frac{2^{4/N}}{S}+\varepsilon\right),
\end{equation*}
which contradicts our assumption $\Sigma_{\nu}(\Omega)<S/{2^{4/N}}$. Therefore, we conclude that $u\not\equiv 0$.
 
\indent\emph{Step two: Strong convergence in $L^{2^{*}}(\Omega)$.} By Vitali theorem, 
 \begin{align*}
 \int_{\Omega}|u_k|^{2^{*}}\, dx-\int_{\Omega}|u_k-u|^{2^{*}}\, dx &=-\int_{\Omega}\int^{1}_{0}\frac{d}{dt}|u_k-tu|^{2^{*}}\, dt\, dx \nonumber\\
 &=2^{*}\int_{\Omega}\int^{1}_{0}u(u_k-tu)|u_k-tu|^{2^{*}-2}\, dt\, dx \nonumber\\
 &=2^{*}\int_{\Omega}\int^{1}_{0}u(u-tu)|u-tu|^{2^{*}-2}\, dt\, dx+o(1)\nonumber\\
 &=\int_{\Omega}|u|^{2^{*}}\, dx+o(1).
 \end{align*}
Since $(u_k)_{k\in\N}\subset M_{\Omega}$,
\begin{equation}\label{vitali-2}
1- \int_{\Omega}|u_k-u|^{2^{*}}\, dx= \int_{\Omega}|u|^{2^{*}}\, dx+o(1).
\end{equation} 
By weak convergence in $H^{2}_{\nu}(\Omega)$,
 \begin{align}
 J(u_k)&=J(u_k-u)+2\langle u_k -u, u\rangle_{H^{2}_{\nu}(\Omega)}+J(u)\nonumber\\
 &=J(u_k-u)+J(u)+o(1),\label{H2-energy-min}
 \intertext{and by strong convergence in $H^1(\Omega)$,}
 J(u_k)&=\int_{\Omega}|\Delta u_k-\Delta u|^{2}\, dx+ J(u)+o(1)\label{E-energy-min}.
 \end{align}
From Step one,
 \begin{equation*}
\widetilde{u}:=\frac{u}{\|u\|_{L^{2^{*}}(\Omega)}}\in M_{\Omega}
 \end{equation*}
and since $J(\widetilde{u})\geq \Sigma_{\nu}(\Omega)$,
\begin{equation}\label{minonM}
J(u)\geq \Sigma_{\nu}(\Omega)\|u\|^{2}_{L^{2^{*}}(\Omega)}.
\end{equation}
Thus, 
 \begin{align*}
 \Sigma_{\nu}(\Omega)&=\int_{\Omega}|\Delta u_k-\Delta u|^{2}\, dx+J(u)+o(1)\hspace{.5cm}\text{( by \eqref{E-energy-min})}\\
 &\geq\left(\frac{2^{4/N}}{S}+\varepsilon\right)^{-1}\|u_k-u\|^{2}_{L^{2^{*}}(\Omega)}+J(u)+o(1)\hspace{.2cm} \text{( by \eqref{DHL-Lemma} and \eqref{H1-convergence})}\\
&\geq\left(\frac{2^{4/N}}{S}+\varepsilon\right)^{-1}\|u_k-u\|^{2}_{L^{2^{*}}(\Omega)}+\Sigma_{\nu}(\Omega)\|u\|^{2}_{L^{2^{*}}(\Omega)}+o(1)\hspace{.5cm}\text{( by \eqref{minonM})}\\
 &=\left[\left(\frac{2^{4/N}}{S}+\varepsilon\right)^{-1}-\Sigma_{\nu}(\Omega)\right]\|u_k-u\|^{2}_{L^{2^{*}}(\Omega)}+\Sigma_{\nu}(\Omega)\left(\|u_k-u\|^{2}_{L^{2^{*}}(\Omega)}+\|u\|^{2}_{L^{2^{*}}(\Omega)}\right)+o(1)\\
 &\geq\left[\left(\frac{2^{4/N}}{S}+\varepsilon\right)^{-1}-\Sigma_{\nu}(\Omega)\right]\|u_k-u\|^{2}_{L^{2^{*}}(\Omega)}+\Sigma_{\nu}(\Omega)\left(\|u_k-u\|^{2^{*}}_{L^{2^{*}}(\Omega)}+\|u\|^{2^{*}}_{L^{2^{*}}(\Omega)}\right)^{\frac{2}{2^{*}}}+o(1)\\
&=\left[\left(\frac{2^{4/N}}{S}+\varepsilon\right)^{-1}-\Sigma_{\nu}(\Omega)\right]\|u_k-u\|^{2}_{L^{2^{*}}(\Omega)}+\Sigma_{\nu}(\Omega)+o(1)\hspace{.5cm}\text{( by \eqref{vitali-2})}. 
 \end{align*}
Since by assumption we have $\Sigma_{\nu}(\Omega)<S/{2^{4/N}}$, we deduce that $u_k\to u$ strongly in $L^{2^{*}}(\Omega)$, and $\|u\|_{L^{2^{*}}(\Omega)}=1$.
 
\indent\emph{Step three: Strong convergence in $H^{2}_{\nu}(\Omega)$.} By weak lower semi-continuity of $J$, and since $u\in M_{\Omega}$,
 \begin{equation*}
 \Sigma_{\nu}(\Omega)\leq J(u)\leq\liminf_{k\to\infty}J(u_k)=\Sigma_{\nu}(\Omega).
 \end{equation*} 
 Therefore, combining this with \eqref{H2-energy-min} we conclude that $u_k\to u$ strongly in $H^{2}_{\nu}(\Omega)$. 
\end{proof}
\section{Asymptotic estimates}\label{sec-estimates-4}
In this section we take into account the smoothness of the boundary $\partial\Omega$ and the effect of the principal curvatures at some boundary point. 

Before proceeding any further, we settle the geometrical aspect of our problem. Since $\Omega\subset\R^N$ is a bounded set, there exists a ball of radius $R_0>0$ such that $\Omega\subset B_{R_0}$. In view of the smoothness of $\Omega$, there exists $\overline{x}\in\partial\Omega$ such that in a neighborhood of $\overline{x}$, we have that $\Omega$ lies on one side of the tangent plane at $\overline{x}$, and the mean curvature with respect to the unit outward normal at $\overline{x}$ is positive. Due to the invariance of rotations and translations, by a change of variables we may assume that $\overline{x}$ is the origin, that the tangent hyperplane coincides with $\{x_N=0\}$, and that $\Omega\subset\R^{N}_{+}=\{x=(x^{\prime},x_N):x_N>0\}$. By the fact that $\Omega$ is a smooth subset, there are $R>0$, and a smooth function $\rho:\{x^{\prime}\in\R^{N-1}:|x^{\prime}|<R\}\to\R_{+}$ such that
 \begin{equation*}
   \begin{cases}
     \Omega\cap B_R=\{ (x^{\prime},x_N)\in B_R:x_N>\rho(x^{\prime})\}\\[.5em]
     \partial\Omega\cap B_R=\{ (x^{\prime},x_N)\in B_R:x_N=\rho(x^{\prime})\}.
     \end{cases}
\end{equation*}
Since the curvature is positive at the origin, there are real constants $(\kappa_j)^{N-1}_{j=1}$, which are called the principal curvatures, that satisfy
\begin{equation}\label{rho-def}
H_N(0):=\frac{2}{N-1}\sum^{N-1}_{j=1}\kappa_j>0,\hspace{.5cm}\text{and}\hspace{.5cm}\rho(x^{\prime})=\sum^{N-1}_{j=1}\kappa_j x^{2}_{j}+O(|x^{\prime}|^{3})\hspace{.3cm}\text{as}\ |x^{\prime}|\to 0.
\end{equation}

Recall that a crucial point in getting compactness in the proof of Lemma \ref{S_alphaK} was the assumption $\Sigma_{\nu}(\Omega)<S/{2^{N/4}}$. In our next result we establish this inequality. 

\begin{lemma}\label{sharp-S-alpha}
Assume that $\Omega$ is an open bounded subset of $\R^N$ with smooth boundary, and $N\geq 5$. Then, there holds
\begin{equation*}
\Sigma_{\nu}(\Omega)<\frac{S}{2^{4/N}}.
\end{equation*}
\end{lemma}
\begin{proof}
\noindent\emph{Step one: Straightening the boundary.} Note that for any $x\in\partial\Omega\cap B_R$, we have that $x=(x^{\prime},\rho(x^{\prime}))$, where $\rho$ is defined in \eqref{rho-def}. Consequently, an outward orthogonal vector to the tangent space is given by
\begin{equation*}
\nu(x)=\begin{bmatrix}
    \nabla\rho(x^{\prime})\\
    -1
\end{bmatrix}
.
\end{equation*}
For some open subset $V$ of $\R^N$, we define
\begin{align}\label{Phi-chart-def}
  \Phi \colon V\subset\R^N &\to \R^N\nonumber\\
  (y^{\prime},y_N) &\mapsto (y^{\prime},\rho(y^{\prime}))-y_N\nu(y^{\prime},\rho(y^{\prime})).
\end{align}
Observe that the Jacobian matrix of $\Phi$ is given by
\begin{equation}\label{Jacobian-Phi}
D\Phi=\begin{bmatrix}
    1-y_N\dfrac{\partial\nu_1}{\partial{y_1}} &-y_N\dfrac{\partial\nu_1}{\partial y_2} & \dots & -y_N\dfrac{\partial\nu_1}{\partial y_{N-1}}&-\nu_1\\[.7em]
    -y_N\dfrac{\partial\nu_2}{\partial y_1}& 1-y_N\dfrac{\partial\nu_2}{\partial y_2} & \dots &-y_N\dfrac{\partial\nu_2}{\partial y_{N-1}}& -\nu_2\\[.7em]
    \vdots & \vdots & \ddots & \vdots &\vdots\\[.7em]
    -y_N\dfrac{\partial\nu_{N-1}}{\partial y_1}& -y_N\dfrac{\partial\nu_{N-1}}{\partial y_2}& \dots &1-y_N\dfrac{\partial\nu_{N-1}}{\partial y_{N-1}}&-\nu_{N-1}\\[.7em]
    \nu_1 & \nu_2& \dots &\nu_{N-1} & 1
\end{bmatrix}
.
\end{equation}
From this, we immediately deduce that for $(y^{\prime},y_N)=(0,0)$, there holds $D\Phi(0,0)=\text{Id}$, where $\text{Id}$ is the identity matrix of size $N$. By the Inverse Function Theorem there exist $r_0>0$, and $U$ an open subset of $\R^N$ such that $\Phi:B^{+}_{r_0}\to\Omega\cap U$ is a smooth diffeomorphism, where $B^{+}_{r_0}:=B_{r_0}\cap\{y_N> 0\}$. Now, let $\eta$ be a $C^{\infty}$ radial fixed cut-off function with $0\leq\eta\leq 1$, and
\begin{align*}
\eta(r)=
\begin{cases}
1,&\text{if}\ r\leq r_{0}/4\\
0,&\text{if}\ r\geq r_{0}/2.
\end{cases}
\end{align*}
Set 
\begin{equation*}
\varphi_{\varepsilon}(y):=\eta(|y|) u_{\varepsilon}(|y|),
\end{equation*}
where $u_{\varepsilon}$ is defined in \eqref{U-epsi}. As a consequence, the following function is well-defined
\begin{equation*}
\psi_{\varepsilon}(x):=\varphi_{\varepsilon}\circ\Phi^{-1}(x).
\end{equation*}
Note that for $x=(x^{\prime},\rho(x^{\prime}))\in\partial\Omega\cap B_R$, 
\begin{align*}
\lim_{t\to 0}\frac{\psi_{\varepsilon}(x)-\psi_{\varepsilon}(x-t\nu(x))}{t}&=\lim_{t\to 0}\frac{\varphi_{\varepsilon}(\Phi^{-1}(x))-\varphi_{\varepsilon}(\Phi^{-1}(x-t\nu(x)))}{t}\\
&=\lim_{t\to 0}\frac{\varphi_{\varepsilon}(y^{\prime},0)-\varphi_{\varepsilon}(y^{\prime},t)}{t}\\
&=-\partial_{y_N}\varphi_{\varepsilon}(y^{\prime},0)\\
&=0,
\end{align*}
where in the last equality we have used the fact that $\varphi_{\varepsilon}$ is a radial function. Therefore $\psi_{\varepsilon}$ belongs to $H^{2}_{\nu}(\Omega)$.

By \eqref{rho-def} and \eqref{Jacobian-Phi},
\begin{equation}\label{matrix-Phi-diff}
D\Phi(y^{\prime},y_N)=\Id+A(y^{\prime},y_N)+O(|(y^{\prime},y_N)|^2),
\end{equation}
where
\begin{equation*}
A(y^{\prime},y_N)=\begin{bmatrix}
    -2y_N\kappa_1 &0 & \dots &0 & -2\kappa_1 y_1\\[.5em]
    0& -2y_N\kappa_2 & \dots &0 & -2\kappa_2 y_2\\[.5em]
    \vdots & \vdots & \ddots & \vdots &\vdots\\[.5em]
    0& 0& \dots & -2y_N\kappa_{N-1}&-2\kappa_{N-1} y_{N-1}\\[.5em]
     2\kappa_1 y_1& 2\kappa_2 y_2& \dots & 2\kappa_{N-1} y_{N-1}& 0
\end{bmatrix}
.
\end{equation*}
In addition, since $D\Phi$ is an inversible matrix, 
\begin{equation*}
D\Phi^{-1}(x)=\left(D\Phi(y)\right)^{-1},
\end{equation*}
where $x=\Phi(y)$. Thus, 
\begin{equation}\label{matrix-Phi-diff-II}
D\Phi^{-1}(x)=\Id-A(y)+O(|y|^2).
\end{equation}
Henceforth, for convenience we write
\begin{equation*}
y=\Phi^{-1}(x),\hspace{.3cm}\text{and}\hspace{.3cm} y_j=\left(\Phi^{-1}(x)\right)_j\hspace{.3cm}\text{for}\ j=1,\dots,N.
\end{equation*}
In view of the above notation, by \eqref{matrix-Phi-diff-II} the elements $(D\Phi^{-1}(x))_{ij}$ of the matrix $D\Phi^{-1}(x)$ are given by 
\begin{align}\label{inverse-Phi-system}
\begin{cases}
\dfrac{\partial y_{j}}{\partial x_j}&=1+2y_N\kappa_j + O(|y|^2),\ j\in\{1,\dots,N-1\},\\[1em]
\dfrac{\partial y_{i}}{\partial x_N}&=2\kappa_i y_i + O(|y|^2),\ i\in\{1,\dots,N-1\},\\[1em]
\dfrac{\partial y_N}{\partial x_j}&=-2\kappa_j y_j + O(|y|^2),\ j\in\{1,\dots,N-1\},\\[1em]
\dfrac{\partial y_i}{\partial x_j}&=O(|y|^2),\ i\neq j\ \text{and}\  i,j\in\{1,\dots,N-1\},\\[1em]
\dfrac{\partial y_N}{\partial x_N}&=1+ O(|y|^2).
\end{cases}
\end{align}
Now using the chain rule, for any $j\in\{1,\dots,N\}$,
\begin{align}
\frac{\partial\psi_{\varepsilon}(x)}{\partial x_j}&=\sum^{N}_{l=1}\frac{\partial\varphi_{\varepsilon}(y)}{\partial y_l}\frac{\partial y_l}{\partial x_j},\label{grad-chart-y}
\intertext{and}
\frac{\partial^2\psi_{\varepsilon}(x)}{\partial x^{2}_{j}}&=\sum^{N}_{k,l=1}\frac{\partial^2\varphi_{\varepsilon}(y)}{\partial y_k\partial y_l}\left(\frac{\partial y_k}{\partial x_j}\right)\left(\frac{\partial y_l}{\partial x_j}\right)+\sum^{N}_{l=1}\frac{\partial\varphi_{\varepsilon}(y)}{\partial y_l}\frac{\partial^2 y_l}{\partial x^{2}_{j}}.\label{lapla-chart-y}
\end{align}

\indent\emph{Step two: Estimate for $\|\Delta\psi_{\varepsilon}\|^2_{L^2(\Omega)}$.} We begin by  computing the derivatives of $u_{\varepsilon}$. Notice that for $l,k\in\{1,\dots,N\}$ fixed, 
\begin{align*}
\frac{\partial u_{\varepsilon}(y)}{\partial y_l}&=-\frac{\gamma_N (N-4)\varepsilon^{\frac{N-4}{2}}}{(\varepsilon^2+|y|^2)^{\frac{N-2}{2}}}y_l,
\intertext{and}
\frac{\partial^2 u_{\varepsilon}(y)}{\partial y_k\partial y_l}&=\gamma_N (N-4)\varepsilon^{\frac{N-4}{2}}\left(\frac{-\delta_{lk}}{(\varepsilon^2+|y|^2)^{\frac{N-2}{2}}}+\frac{(N-2)y_k y_l}{(\varepsilon^2+|y|^2)^{\frac{N}{2}}}\right),
\end{align*}
where $\delta_{lk}$ is the Kronecker delta, that is, $\delta_{lk}=1$ if $l=k$, and $\delta_{lk}=0$ otherwise. Since we are interested in an estimate for the $L^2$-norm of $\Delta\psi_{\varepsilon}$, it is enough to compute the derivatives of $\psi_{\varepsilon}$ when $\varphi_{\varepsilon}\equiv u_{\varepsilon}$. In this situation, from \eqref{inverse-Phi-system} for $j\in\{1,\dots,N-1\}$,
\begin{align}\label{psi-lapla-y_j}
\frac{\partial^2 \psi_{\varepsilon}(x)}{\partial x^{2}_{j}}&=\frac{\partial^2 u_{\varepsilon}(y)}{\partial y^{2}_{j}}\left(\frac{\partial y_j}{\partial x_j}\right)^2+2\frac{\partial^2 u_{\varepsilon}(y)}{\partial y_j \partial y_N}\left(\frac{\partial y_j}{\partial x_j}\right)\left(\frac{\partial y_N}{\partial x_j}\right)+\frac{\partial u_{\varepsilon}(y)}{\partial y_N}\frac{\partial^2 y_N}{\partial x^{2}_{j}}+O\left(\frac{\varepsilon^{\frac{N-4}{2}}|y|^2}{(\varepsilon^2+|y|^2)^{\frac{N-2}{2}}}\right)\nonumber\\
&=\frac{\partial^2 u_{\varepsilon}(y)}{\partial y^{2}_{j}}(1+4y_N\kappa_j)-4\kappa_j y_j\frac{\partial^2 u_{\varepsilon}(y)}{\partial y_j \partial y_N} -2\kappa_j\frac{\partial u_{\varepsilon}(y)}{\partial y_N} +O\left(\frac{\varepsilon^{\frac{N-4}{2}}}{(\varepsilon^2+|y|^2)^{\frac{N-4}{2}}}\right)\nonumber\\
&=\frac{\partial^2 u_{\varepsilon}(y)}{\partial y^{2}_{j}}-\frac{2\gamma_N (N-4)\varepsilon^{\frac{N-4}{2}}}{(\varepsilon^2+|y|^2)^{\frac{N-2}{2}}}y_N\kappa_j+O\left(\frac{\varepsilon^{\frac{N-4}{2}}}{(\varepsilon^2+|y|^2)^{\frac{N-4}{2}}}\right).
\end{align}
In case $j=N$, 
\begin{align}\label{psi-lapla-y_N}
\frac{\partial^2\psi_{\varepsilon}(x)}{\partial x^{2}_{N}}&=\frac{\partial^2 u_{\varepsilon}(y)}{\partial y^{2}_N}+2\sum^{N-1}_{j=1}\frac{\partial^2 u_{\varepsilon}(y)}{\partial y_j \partial y_N}\left(\frac{\partial y_j}{\partial x_N}\right)\left(\frac{\partial y_N}{\partial x_N}\right)+O\left(\frac{\varepsilon^{\frac{N-4}{2}}}{(\varepsilon^2+|y|^2)^{\frac{N-4}{2}}}\right)\nonumber\\[.5em]
&=\frac{\partial^2 u_{\varepsilon}(y)}{\partial y^{2}_N}+4\gamma_N (N-4)(N-2)\sum^{N-1}_{j=1}\frac{\kappa_j y^{2}_{j} y_N\varepsilon^{\frac{N-4}{2}}}{(\varepsilon^2+|y|^2)^{\frac{N}{2}}}+O\left(\frac{\varepsilon^{\frac{N-4}{2}}}{(\varepsilon^2+|y|^2)^{\frac{N-4}{2}}}\right).
\end{align}
Thus, by \eqref{psi-lapla-y_j} and \eqref{psi-lapla-y_N},
\begin{align*}
\Delta\psi_{\varepsilon}(x)&=\Delta u_{\varepsilon}(y)-\frac{\gamma_N (N-4)(N-1)H_{N}(0)\varepsilon^{\frac{N-4}{2}}}{(\varepsilon^2+|y|^2)^{\frac{N-2}{2}}}y_N\\[.5em]
&\qquad +4\gamma_N (N-4)(N-2)\sum^{N-1}_{j=1}\frac{\kappa_j y^{2}_{j} y_N\varepsilon^{\frac{N-4}{2}}}{(\varepsilon^2+|y|^2)^{\frac{N}{2}}}+O\left(\frac{\varepsilon^{\frac{N-4}{2}}}{(\varepsilon^2+|y|^2)^{\frac{N-4}{2}}}\right).
\end{align*}
Recall that
\begin{equation*}
\Delta u_{\varepsilon}(y)=-\frac{\gamma_N (N-4)\varepsilon^{\frac{N-4}{2}}(N\varepsilon^2+2|y|^2)}{(\varepsilon^2+|y|^2)^{\frac{N}{2}}}<0.
\end{equation*}
Observe that due to the form of the matrix $D\Phi(y)$,
\begin{equation*}
\det D\Phi(y)=1-(N-1)H_{N}(0)y_N+O(|y|^2).
\end{equation*}
Henceforth, for convenience we denote
\begin{equation}\label{d_N-definition}
d_N=\gamma^{2}_N(N-4)^2 (N-1).
\end{equation}
Thus,
\begin{align*}
\int_{\Omega}|\Delta\psi_{\varepsilon}(x)|^2\, dx&=\int_{\Omega\cap U}|\Delta\psi_{\varepsilon}(x)|^2\, dx+O(\varepsilon^{N-4})\\
&=\int_{B^{+}_{r_0}}|\Delta\psi_{\varepsilon}(\Phi(y))|^2 |\det D\Phi(y)|\, dy+O(\varepsilon^{N-4})\\
&=I_1+I_2+I_3+I_4+I_5,
\end{align*}
where
\begin{align*}
I_1&=\int_{B^{+}_{{r_0}/2}}|\Delta u_{\varepsilon}(y)|^2\, dy,\\
I_2&=-d_N H_N(0)\varepsilon^{N-4}\int_{B^{+}_{{r_0}/2}}\frac{(N\varepsilon^2+2|y|^2)^2}{(\varepsilon^2+|y|^2)^N}y_N\, dy,\\
I_3&=2 d_N H_N(0)\varepsilon^{N-4}\int_{B^{+}_{{r_0}/2}}\frac{(N\varepsilon^2+2|y|^2)}{(\varepsilon^2+|y|^2)^{N-1}}y_N\, dy,\\
I_4&=-\frac{8 d_N (N-2)}{N-1}\varepsilon^{N-4}\sum^{N-1}_{j=1}\kappa_j\int_{B^{+}_{{r_0}/2}}\frac{(N\varepsilon^2+2|y|^2)}{(\varepsilon^2+|y|^2)^{N}}y^{2}_{j} y_N\, dy,\\
\intertext{and}
I_5&=
\begin{cases}
O(\varepsilon^2), &\text{if}\ N\geq 7\\
O(\varepsilon^2 \log\frac{1}{\varepsilon}), &\text{if}\ N=6\\
O(\varepsilon), &\text{if}\ N=5.
\end{cases}
\end{align*}
In this way, by \eqref{laplacian-o},
\begin{equation*}
I_1=\frac{S^{N/4}}{2}+O(\varepsilon^{N-4}),
\end{equation*}
and we estimate $I_2+I_3$ as follows,
\begin{align*}
I_2+I_3&=-d_N H_{N}(0)(N-2)\varepsilon^{N-2}\int_{B^{+}_{{r_0}/2}}\frac{(N\varepsilon^2+2|y|^2)}{(\varepsilon^2+|y|^2)^N}y_N\, dy\nonumber\\
&=-d_N H_{N}(0)(N-2)J_1\varepsilon+o(\varepsilon),
\end{align*}
where
\begin{equation*}
J_1=\int_{\R^{N}_{+}}\frac{N+2|y|^2}{(1+|y|^2)^{N}}y_N\, dy>0.
\end{equation*}
To estimate $I_4$ we first note that by symmetry,
\begin{equation*}
\int_{B^{+}_{r_0/2}}\frac{N\varepsilon^2+2|y|^2}{(\varepsilon^2+|y|^2)^N}y^{2}_{j}y_N\, dy=\frac{1}{N-1}\int_{B^{+}_{r_0/2}}\frac{(N\varepsilon^2+2|y|^2)(|y|^2-y^{2}_{N})}{(\varepsilon^2+|y|^2)^N}y_N\, dy.
\end{equation*}
Thus, 
\begin{align*}
I_4&=-\frac{4d_N (N-2)}{N-1}H_N(0)\varepsilon^{N-4}\int_{B^{+}_{r_0/2}}\frac{(N\varepsilon^2+2|y|^2)(|y|^2-y^{2}_{N})}{(\varepsilon^2+|y|^2)^N}y_N\, dy\\
&=-\frac{4d_N (N-2)}{N-1}H_N(0)\varepsilon\int_{B^{+}_{r_0/2\varepsilon}}\frac{(N+2|y|^2)(|y|^2-y^{2}_{N})}{(1+|y|^2)^N}y_N\, dy\\
&=-\begin{cases}
\frac{4d_N (N-2)}{N-1}H_N(0)J_2 \varepsilon+o(\varepsilon),&\text{if}\ N\geq 6\\[.4em]
8\pi^{2}\sqrt[4]{105}H_{5}(0)\varepsilon\log\frac{1}{\varepsilon}+O(\varepsilon),&\text{if}\ N=5,
\end{cases}
\end{align*}
where 
\begin{equation*}
J_2=\int_{\R^{N}_{+}}\frac{(N+2|y|^2)(|y|^2-y^{2}_N)}{(1+|y|^2)^N}y_N\, dy>0.
\end{equation*}
Consequently, 
\begin{equation}\label{Lapla-Fermi-estim}
\int_{\Omega}|\Delta\psi_{\varepsilon}|^2\, dx=\frac{S^{N/4}}{2}+o(\varepsilon)-
\begin{cases}
d_N H_N(0)(N-2)\left(J_1+\dfrac{4 J_2}{N-1}\right)\varepsilon+o(\varepsilon),&\text{if}\ N\geq 6\\[.4em]
8\pi^{2}\sqrt[4]{105}H_5(0)\varepsilon\log\frac{1}{\varepsilon}+O(\varepsilon),&\text{if}\ N=5.
\end{cases}
\end{equation}

\indent\emph{Step three: Estimate for $\|\psi_{\varepsilon}\|^{2^{*}}_{{L^{2^{*}}}(\Omega)}$.} Arguing as in the previous step, 
\begin{align}\label{Step-3-estim-critical}
\int_{\Omega}|\psi_{\varepsilon}(x)|^{2^{*}}\, dx&=\int_{\Omega\cap U}|\psi_{\varepsilon}(x)|^{2^{*}}\, dx+O(\varepsilon^N)\nonumber\\
&=\int_{B^{+}_{r_0}}|\psi_{\varepsilon}(\Phi(y))|^{2^{*}}|\det D\Phi(y)|\, dy+O(\varepsilon^N)\nonumber\\
&=\int_{B^{+}_{{r_0}/2}}|u_{\varepsilon}(y)|^{2^{*}}\big(1-(N-1)H_{N}(0)y_N+O(|y|^2)\big)\, dy+O(\varepsilon^N)\nonumber\\
&=\frac{S^{N/4}}{2}+O(\varepsilon^2)-\gamma^{2^{*}}_N (N-1)H_{N}(0)\varepsilon^N\int_{B^{+}_{{r_0}/2}}\frac{y_N}{(\varepsilon^2+|y|^2)^{N}}\, dy\nonumber\\
&=\frac{S^{N/4}}{2}+O(\varepsilon^2)-\gamma^{2^{*}}_N (N-1)H_{N}(0)J_3\varepsilon,
\end{align}
where
\begin{equation*}
J_3=\int_{\R^{N}_{+}}\frac{y_N}{(1+|y|^2)^N}\, dy>0.
\end{equation*}

\indent\emph{Step four: Estimate for $\|\nabla\psi_{\varepsilon}\|^2_{L^2(\Omega)}$.} Arguing as previously, 
\begin{equation}\label{grad-chart-y-est}
\int_{\Omega}|\nabla\psi_{\varepsilon}(x)|^2\, dx\leq
\begin{cases}
    O(\varepsilon^2),&\text{if}\ N\geq 7\\[.5em]
    O\left(\varepsilon^2\log\frac{1}{\varepsilon}\right),&\text{if}\ N=6\\[.5em]
    O(\varepsilon),&\text{if}\ N=5. 
\end{cases}
\end{equation}

\indent\emph{Step five: Estimate for $\|\psi_{\varepsilon}\|^2_{L^2(\Omega)}$.} In the same way, 
\begin{equation}\label{L2-norm-chart-y-est}
\int_{\Omega}|\psi_{\varepsilon}(x)|^2\, dx\leq
\begin{cases}
    O(\varepsilon^4),&\text{if}\ N\geq 9\\[.5em]
    O\left(\varepsilon^4\log\frac{1}{\varepsilon}\right),&\text{if}\ N=8\\[.5em]
    O(\varepsilon^{N-4}),&\text{if}\ N\in\{5,6,7\}. 
\end{cases}
\end{equation}

\indent\emph{Step six: Conclusion.} By \eqref{Lapla-Fermi-estim}-\eqref{L2-norm-chart-y-est}, for $N\geq 6$, 
\begin{align*}
\Sigma_{\nu}(\Omega)&\leq
\frac{\int_{\Omega}|\Delta\psi_{\varepsilon}|^2\, dx+\int_{\Omega}|\nabla\psi_{\varepsilon}|^2\, dx+\alpha\int_{\Omega}|\psi_{\varepsilon}|^2\, dx}{\left(\int_{\Omega}|\psi_{\varepsilon}|^{2^{*}}\, dx\right)^{\frac{2}{2^{*}}}}\\
&\leq\frac{S}{2^{4/N}}+o(\varepsilon)\\
&\qquad -2^{1-4/N}S^{1-N/4}H_N(0)\Bigg[d_N(N-2)\left(J_1+\frac{4J_2}{N-1}\right)-\frac{(N-4)(N-1)}{N}\gamma^{2^{*}}_N J_3\Bigg]\varepsilon.
\end{align*}
Now recall the explicit values of $\gamma_N$, and $d_N$ in \eqref{Gamma-def-opt} and \eqref{d_N-definition}, respectively. In order to show that the term between the brackets is positive, it is enough to guarantee that 
\begin{equation*}
\beta_N:=\frac{1}{N+2}\int_{\R^{N}_+}\frac{N+2|y|^2}{(1+|y|^2)^N}y_N\, dy-\int_{\R^{N}_+}\frac{y_N}{(1+|y|^2)^N}\, dy\hspace{.2cm}\text{is positive}.
\end{equation*}
Denote by $\mathbb{S}^{N-1}$ the unit sphere and by $c(N)$ a positive constant that depends on $N$. Then,
\begin{align*}
\beta_N&=\int_{\mathbb{S}^{N-1}\cap\R^{N}_{+}}y_N\, d\sigma\left(\frac{1}{N+2}\int^{\infty}_{0}\frac{N+2r^2}{(1+r^2)^N}r^{N}\, dr-\int^{\infty}_{0}\frac{r^N}{(1+r^2)^N}\, dr\right)\\
&=c(N)\int^{\infty}_{0}\frac{r^2-1}{(1+r^2)^N}r^N\, dr\\
&=c(N)\left(\int^{\infty}_{0}\frac{t^{\frac{N+1}{2}}}{(1+t)^N}\, dt-\int^{\infty}_{0}\frac{t^{\frac{N-1}{2}}}{(1+t)^N}\, dt \right)\\
&=c(N)\left[\frac{\Gamma\left(\frac{N+1}{2}+1\right)\Gamma\left(\frac{N-1}{2}-1\right)-\Gamma\left(\frac{N+1}{2}\right)\Gamma\left(\frac{N-1}{2}\right)}{\Gamma(N)}\right]\\
&=c(N)\frac{\Gamma\left(\frac{N-3}{2}\right)\Gamma\left(\frac{N+1}{2}\right)}{\Gamma(N)},
\end{align*}
which yields $\beta_N>0$. Now, going back to the above inequality, and making $\varepsilon$ sufficiently small we get our result for $N\geq 6$. 

In case $N=5$, 
\begin{align*}
\Sigma_{\nu}(\Omega)&\leq\frac{S}{2^{4/N}}+O(\varepsilon)-2^{14/5}\pi^{2}\sqrt[4]{\frac{105}{S}}H_5(0)\varepsilon\log\frac{1}{\varepsilon}\\
&<\frac{S}{2^{4/N}},
\end{align*}
provided $\varepsilon$ is sufficiently small. This completes the proof.
\end{proof}

Now we are in position to give the proof of Theorem \ref{main-theo-I}.

\begin{proof}[Proof of Theorem \ref{main-theo-I}]
By Lemmas \ref{S_alphaK} and \ref{sharp-S-alpha}, there exists a minimizer $u\in M_{\Omega}$ for $\Sigma_{\nu}(\Omega)$. Now, we have to rule out $u$ as the constant solution $u_1=\alpha^{\frac{N-4}{8}}$. To this end, note that
\begin{equation*}
\frac{\int_{\Omega}(|\Delta u_{1}|^2+|\nabla u_1|^2+\alpha |u_1|^2)\, dx}{\left(\int_{\Omega}|u_1|^{2^{*}}\, dx\right)^{\frac{2}{2^{*}}}}=\alpha |\Omega|^{4/N},
\end{equation*}
where $|\Omega|$ stands for the Lebesgue measure of $\Omega$. Then, we are done if we have $\overline{\alpha}>0$ for which 
\begin{equation*}
 \alpha |\Omega|^{4/N}>\Sigma_{\nu}(\Omega),\hspace{.5cm}\text{for all}\ \alpha\geq\overline{\alpha}.
\end{equation*}
By Lemma \ref{sharp-S-alpha}, the above inequality follows by taking $\overline{\alpha}=S\slash (2 |\Omega|)^{4/N}$. This completes the proof.
\end{proof}
\section{A Sobolev inequality of second order}\label{Section-2nd-Sob-ineq}
Our aim in this section is to prove Lemma \ref{sharp-H2-mu}. Our approach consists in providing a sharp inequality in $H^{2}_{\nu}(\R^{N}_{+})$, and then by a partition of unity argument we establish our result for functions in $H^{2}_{\nu}(\Omega)$.

\begin{proof}[proof of Lemma \ref{sharp-H2-mu}]
\noindent\emph{Step one: There holds $\Sigma_{\nu}(\R^{N}_{+})=S/2^{4/N}$, and the infimum is not achieved.} Consider the function $z_{\varepsilon}$ defined in \eqref{z-epsi-sharp-half}. By symmetry,  \eqref{estimate-lapla-z}, \eqref{L2-critical-L2}, and \eqref{estimate-grad-z} we infer that $z_{\varepsilon}\in H^{2}_{\nu}(\R^{N}_{+})$ satisfies
\begin{align*}
\int_{\R^{N}_{+}}|\Delta z_{\varepsilon}(|x|)|^2\, dx&=\frac{S^{N/4}}{2}+o(1),\hspace{1cm}\int_{\R^{N}_{+}}|\nabla z_{\varepsilon}(|x|)|^2\, dx=o(1),\\
\int_{\R^{N}_{+}}|z_{\varepsilon}(|x|)|^2\, dx&=o(1),\hspace{1cm}\text{and}\hspace{1cm}\int_{\R^{N}_{+}}|z_{\varepsilon}(|x|)|^{2^{*}}\, dx=\frac{S^{N/4}}{2}+o(1).
\end{align*}
As a consequence, 
\begin{equation*}
\lim_{\varepsilon\to 0}\frac{\int_{\R^{N}_{+}}(|\Delta z_{\varepsilon}|^2+|\nabla z_{\varepsilon}|^2+\alpha|z_{\varepsilon}|^2)\, dx}{\left(\int_{\R^{N}_{+}}|z_{\varepsilon}|^{2^{*}}\, dx\right)^{\frac{2}{2^{*}}}}=\frac{S}{2^{4/N}},
\end{equation*}
which shows that
\begin{equation}\label{2nd-ineq-Sob-I}
\Sigma_{\nu}(\R^{N}_{+})\leq \frac{S}{2^{4/N}}.
\end{equation}
Now we argue by contradiction, that is, assume that there exists $\phi\in H^{2}_{\nu}(\R^{N}_{+})$ such that 
\begin{equation}\label{2nd-ineq-Sob-II}
\frac{\int_{\R^{N}_{+}}(|\Delta \phi|^2+|\nabla \phi|^2+\alpha|\phi|^2)\, dx}{\left(\int_{\R^{N}_{+}}|\phi|^{2^{*}}\, dx\right)^{\frac{2}{2^{*}}}}\leq\frac{S}{2^{4/N}}.
\end{equation}
Define $\widetilde{\phi}$ as the reflection of $\phi$ with respect to the $x_N$-axis,
\begin{align*}
\widetilde{\phi}(x)=
\begin{cases}
\phi(x^{\prime},x_{N}),&\text{if}\ x_N\geq 0\\
\phi(x^{\prime},-x_{N}),&\text{if}\ x_N<0.
\end{cases}
\end{align*}
Since $\partial_{\nu}\phi=0$ along $\partial\R^{N}_{+}$ it is easily seen that $\widetilde{\phi}$ belongs to $H^{2}(\R^N)$. Then, using the symmetry (doubling the integrals),
\begin{equation*}
\frac{\int_{\R^{N}}(|\Delta \widetilde{\phi}|^2+|\nabla\widetilde{\phi}|^2+\alpha|\widetilde{\phi}|^2)\, dx}{\left(\int_{\R^{N}}|\widetilde{\phi}|^{2^{*}}\, dx\right)^{\frac{2}{2^{*}}}}\leq S.
\end{equation*}
However, this is a contradiction with Lemma \ref{Sigma=S-Rn}. Therefore, there exists no $\phi$ that belongs to $H^{2}_{\nu}(\R^{N}_{+})$ such that \eqref{2nd-ineq-Sob-II} holds. In other words,
\begin{equation*}
\frac{\int_{\R^{N}_{+}}(|\Delta \phi|^2+|\nabla \phi|^2+\alpha|\phi|^2)\, dx}{\left(\int_{\R^{N}_{+}}|\phi|^{2^{*}}\, dx\right)^{\frac{2}{2^{*}}}}>\frac{S}{2^{4/N}},\hspace{.3cm}\text{for all}\hspace{.3cm}\phi\in H^{2}_{\nu}(\R^{N}_{+}).
\end{equation*}
The above inequality combined with \eqref{2nd-ineq-Sob-I} implies that $\Sigma_{\nu}(\R^{N}_{+})=S/2^{4/N}$, and the infimum is not achieved.

\indent\emph{Step two: A partition of unity argument.} Since $\overline{\Omega}$ is a compact set, we can find finitely many points $x_i\in\overline{\Omega}$, radii $r_i>0$, with  corresponding sets $\Omega_i=\Omega\cap B_{r_i}(x_i)$ such that
\begin{equation*}
\overline{\Omega}\subset\bigcup^{n}_{i=1}\Omega_i.
\end{equation*}
Up to increasing the number of open sets,  we can assume that $x_{i}\in\partial\Omega$ whenever ${\Omega_{i}\cap\partial\Omega}\neq\varnothing$. Now let $(\widetilde{\zeta_i})^{n}_{i=1}$ be a smooth partition of unity subordinated to the covering $(\Omega_i)^{n}_{i=1}$. We split the set of indices as 
\begin{equation*}
\{1,2,\dots,n\}=\mathcal{I}\cup\mathcal{J},
\end{equation*}
where $\mathcal{I}$ contains the indices with $x_i \in\Omega$ while $\mathcal{J}$ contains the indices with $x_i\in\partial\Omega$.

\indent\emph{Case one: $\Omega_i\cap\Omega=\varnothing$.} We set
\begin{equation*}
\zeta_i=\frac{\widetilde{\zeta_i}^5}{\sum^{n}_{i=1}\widetilde{\zeta_i}^5}.
\end{equation*}
By construction, $(\zeta_i)^{n}_{i=1}$ is a partition of unity subordinated to the covering $(\Omega_i)^{n}_{i=1}$ such that $\zeta^{1/2}_i\in C^2(\overline{\Omega})$. We denote by $c_1$, and $c_2$ real positive constants such that $|\nabla\zeta^{1/2}_i|\leq c_1$, and $|\Delta\zeta^{1/2}_i|\leq c_2$.

Now choose any function $\phi\in H^2(\R^N)$. Consequently, $\zeta^{1/2}_i\phi\in H^2(\R^N)$, and $\supp(\zeta^{1/2}_i \phi)\subset \Omega_i$. By Lemma \ref{Sigma=S-Rn}, for $\varepsilon_0>0$,
\begin{align*}
\sum_{i\in\mathcal{I}}\left(\int_{\Omega_i}|\zeta^{1/2}_i \phi|^{2^{*}}\, dx\right)^{\frac{2}{2^{*}}}
&\leq\sum_{i\in\mathcal{I}}\left(\int_{\R^N}|\zeta^{1/2}_i \phi|^{2^{*}}\, dx\right)^{\frac{2}{2^{*}}}\\
&\leq\frac{1}{S}\sum_{i\in\mathcal{I}}\int_{\R^N}|\Delta(\zeta^{1/2}_i \phi)|^{2}\, dx\\
&\leq\frac{2^{4/N}}{S}\sum_{i\in\mathcal{I}}\left[\int_{\R^N}\left(\zeta^{1/2}_i |\Delta\phi|+2|\nabla\zeta^{1/2}_i||\nabla\phi|+|\phi| |\Delta\zeta^{1/2}_i|\right)^2\, dx\right]\\
&\leq\frac{2^{4/N}}{S}\left[(1+\varepsilon_0)^2\|\Delta\phi\|^{2}_{L^{2}(\Omega)}+B(\varepsilon_0)\|\phi\|^{2}_{H^1(\Omega)}\right], 
\end{align*}
where in the last inequality we have used Young inequality two times. Note that, for $\varepsilon_0>0$ sufficiently small, 
\begin{equation*}
\frac{2^{4/N}}{S}(1+\varepsilon_0)^2\leq\frac{2^{4/N}}{S}+\varepsilon\hspace{.2cm}\text{for}\hspace{.2cm} \varepsilon>0,
\end{equation*}
so that,
\begin{equation}\label{partition-unity-int-I}
\sum_{i\in\mathcal{I}}\left(\int_{\Omega_i}|\zeta^{1/2}_i \phi|^{2^{*}}\, dx\right)^{\frac{2}{2^{*}}}\leq\left(\frac{2^{4/N}}{S}+\varepsilon\right)\|\Delta\phi\|^{2}_{L^{2}(\Omega)}+B(\varepsilon)\|\phi\|^{2}_{H^1(\Omega)}.
\end{equation}

\indent\emph{Case two: $\Omega_i\cap\partial\Omega\neq\varnothing$.} In this case, for every $i\in\mathcal{J}$ we consider the maps
\begin{equation*}
\Phi^{-1}_i:\Omega_i\cap\partial\Omega\to V_i\subset\R^{N}_{+}
\end{equation*}
as defined in \eqref{Phi-chart-def}, where $V_i$ is some open subset. As previously observed, these maps have the property that in this new coordinate system, any $\phi\in H^{2}_{\nu}(\Omega)$ implies that $(\zeta^{1/2}_i\phi)\circ\Phi_i$ belongs to $H^{2}_{\nu}(\R^{N}_{+})$ for every $i\in\mathcal{J}$. In this way, we may assume
\begin{equation*}
|\det D\Phi(y)|\leq 1+\varepsilon_0
\end{equation*}
for $\varepsilon_0>0$ small enough, otherwise we may rearrange our covering in such a way that the sets $(\Omega)_{i\in\mathcal{J}}$ have smaller sizes. For convenience, we write $\vartheta_i(y)=(\zeta^{1/2}_i\phi)\circ\Phi(y)$. By the previous step, 
\begin{align}\label{partition-J-final}
\sum_{i\in\mathcal{J}}\left(\int_{\Omega_i\cap\partial\Omega}|(\zeta^{1/2}_i\phi)(x)|^{2^{*}}\, dx\right)^{\frac{2}{2^{*}}}&=\sum_{i\in\mathcal{J}}\left(\int_{V_i}|\zeta^{1/2}_i\phi\circ\Phi(y)|^{2^{*}}|\det D\Phi(y)|\, dy\right)^{\frac{2}{2^{*}}}\nonumber\\
&\leq (1+\varepsilon_0)^{\frac{2}{2^{*}}}\sum_{i\in\mathcal{J}}\left(\int_{\R^{N}_{+}}|\vartheta_i(y)|^{2^{*}}\, dy\right)^{\frac{2}{2^{*}}}\nonumber\\
&\leq\frac{2^{4/N}}{S}(1+\varepsilon_0)^{\frac{2}{2^{*}}}\sum_{i\in\mathcal{J}}\Bigg(\int_{\R^{N}_{+}}|\Delta\vartheta_i(y)|^2\, dy\nonumber\\
&\qquad +\int_{\R^{N}_{+}}|\nabla\vartheta_i(y)|^2\, dy+\alpha\int_{\R^{N}_{+}}|\vartheta_i(y)|^{2}\, dy\Bigg). 
\end{align}
Now we recall that by \eqref{inverse-Phi-system} we may find $\varepsilon_1>0$ small enough such that
\begin{equation}\label{Part-unity-Aubin-I}
\begin{cases}
|\Delta\vartheta_i|\leq (1+\varepsilon_1)|\Delta(\zeta^{1/2}_i\phi)|+\varepsilon_1|\nabla(\zeta^{1/2}_i\phi)|+\varepsilon_1|\zeta^{1/2}_i\phi|\\
|\nabla\vartheta_i|\leq(1+\varepsilon_1)|\nabla(\zeta^{1/2}_i\phi)|+\varepsilon_1|\zeta^{1/2}_i\phi|\\
|\vartheta_i|\leq (1+\varepsilon_1)|\zeta^{1/2}_i\phi|.\\
\end{cases}
\end{equation}
In addition, 
\begin{equation}
\begin{cases}\label{Part-unity-Aubin-II}
|\Delta(\zeta^{1/2}_i\phi)|\leq \zeta^{1/2}_i |\Delta\phi|+2|\nabla\zeta^{1/2}_i| |\nabla\phi|+|\phi| |\Delta\zeta^{1/2}_i|\\
|\nabla(\zeta^{1/2}_i\phi)|\leq \zeta^{1/2}_i |\nabla\phi|+|\nabla \zeta^{1/2}_i| |\phi|.
\end{cases}
\end{equation}
Then, by \eqref{Part-unity-Aubin-I}, and \eqref{Part-unity-Aubin-II} together with Young inequality,
\begin{equation}\label{varphi-L2-PU-1}
\sum_{i\in\mathcal{J}}\int_{\R^{N}_{+}}|\Delta\vartheta_i|^2\, dy=(1+\varepsilon_1)^2\int_{\Omega}|\Delta\phi|^2\, dy+B(\varepsilon_1)\|\phi\|^{2}_{H^{1}(\Omega)}.
\end{equation}
Similarly, 
\begin{align}\label{varphi-L2-PU-2}
\sum_{i\in\mathcal{J}}\int_{\R^{N}_{+}}|\nabla\vartheta_i|^2\, dy&\leq B(\varepsilon_1)\|\phi\|^{2}_{H^{1}(\Omega)},
\intertext{and}
\sum_{i\in\mathcal{J}}\int_{\R^{N}_{+}}|\vartheta_i|^2\, dy&\leq(1+\varepsilon_1)^2 \int_{\Omega}|\phi|^2\, dy.\label{varphi-L2-PU-3}
\end{align}
As previously, for $\varepsilon_0,\varepsilon_1>0$ sufficiently small,
\begin{equation*}
\frac{2^{4/N}}{S}(1+\varepsilon_0)^{\frac{2}{2^{*}}}(1+\varepsilon_1)^2 \leq\frac{2^{4/N}}{S}+\varepsilon\hspace{.3cm}\text{for}\hspace{.3cm}\varepsilon>0.
\end{equation*}
Hence, by inserting \eqref{varphi-L2-PU-1}-\eqref{varphi-L2-PU-3} into \eqref{partition-J-final},
\begin{equation*}
\sum_{i\in\mathcal{J}}\left(\int_{\Omega_i\cap\partial\Omega}|\zeta^{1/2}_i\phi|^{2^{*}}\, dx\right)^{\frac{2}{2^{*}}}\leq\left(\frac{2^{4/N}}{S}+\varepsilon\right)\|\Delta\phi\|^{2}_{L^2(\Omega)}+B(\varepsilon)\|\phi\|^{2}_{H^1(\Omega)}.
\end{equation*}
Therefore, by the above inequality together with \eqref{partition-unity-int-I}, for any $\varepsilon>0$,
\begin{align*}
\|\phi\|^{2}_{L^{2^{*}}(\Omega)}&=\|\phi^2\|_{L^{2^{*}/2}(\Omega)}\\
&=\left\|\sum^{n}_{i=1}\zeta_i \phi^2\right\|_{L^{2^{*}/2}(\Omega)}\\
&\leq\sum^{n}_{i=1}\|\zeta_i \phi^2\|_{L^{2^{*}/2}(\Omega)}\\
&=\sum^{n}_{i=1}\|\zeta^{1/2}_i \phi\|^{2}_{L^{2^{*}}(\Omega)}\\
&=\sum_{i\in\mathcal{I}}\left(\int_{\Omega_i}|\zeta^{1/2}_i \phi|^{2^{*}}\, dx\right)^{\frac{2}{2^{*}}}+\sum_{i\in\mathcal{J}}\left(\int_{\Omega_i\cap\partial\Omega}|\zeta^{1/2}_i \phi|^{2^{*}}\, dx\right)^{\frac{2}{2^{*}}}\\
&\leq\left(\frac{2^{4/N}}{S}+\varepsilon\right)\|\Delta\phi\|^{2}_{L^{2}(\Omega)}+B(\varepsilon)\|\phi\|^{2}_{H^1(\Omega)}.
\end{align*}
This completes the proof.
\end{proof}

\section{Minimizing solutions for small $\alpha$}\label{minimizers}

In this section, we prove the rigidity result for minimizing solutions when $\alpha\to 0$. The proof follows almost directly from the one of \cite{AdimYad93} for the second order case. We start with a $L^q$-bound, for $1\le q\le  \frac{N+4}{N-4}$, on positive solutions. The proof easily follows by integrating the equation. 
\begin{lemma}\label{borneL1}
  Any nonnegative solution $u$ of \eqref{problem-P-alpha} satisfies
  \begin{equation*}
   \alpha \int_{\Omega}u\, dx = \int_{\Omega}|u|^{\frac{N+4}{N-4}}\, dx
    \le \alpha^{\frac{N+4}{8}} |{\Omega}|.
  \end{equation*}
\end{lemma}
The bound is clearly sharp. In the subcritical case, one can use elliptic regularity to bootstrap the corresponding estimate to get a better bound or use Gidas-Spruck blow-up technique \cite{GS} to show directly that $u$ converges uniformly to zero as $\alpha\to 0$, see for instance \cite{MR849484,MR3564729}. In the critical case, we need a further hypothesis to improve the bound as shown by the next lemma.

\begin{lemma}\label{borneLinfinity0}
  Assume that $\alpha_k\to 0$ and the sequence $(v_k)_k\subset H^2_\nu(\Omega)$ satisfying
 \begin{equation}\label{vk}
\begin{cases}
\Delta^2 v_k -\Delta v_k+\alpha_k v_k=|v_k|^{\frac{8}{N-4}}v_k,&\mbox{in}\ \Omega,\\
\partial_{\nu} v_k =\partial_{\nu}(\Delta v_k)= 0,&\mbox{on}\ \partial\Omega,
\end{cases}
\end{equation}  
is such that $v_k\ge 0$ and $\sup_k \|v_k\|_{L^q(\Omega)}<\infty$ for some $q> 2^*$. Then $\|v_k\|_{L^{\infty}(\Omega)}\to 0$. 
\end{lemma}
\begin{proof}
Using elliptic regularity, one shows that a $L^q$-bound on $v_k$, with $q=s(N+4)/(N-4)$, gives a $W^{4,s}$ estimate and therefore a $L^{\frac{Ns}{N-4s}}$ bound. This allows to start a bootstrap argument if $s>2N/(N+4)$ and to deduce  an apriori bound in $C^{0,\gamma}(\overline\Omega)$ for some $0<\gamma<1$. Lemma \ref{borneL1} then shows (with a simple interpolation argument) that $v_k$ converges uniformly to $0$ as $k\to\infty$.
\end{proof}

Observe that solutions with a priori finite energy are merely a priori bounded in $L^{2^*}$ so that Lemma \ref{borneLinfinity} cannot be used for those solutions. The next lemma shows that minimizing solutions are bounded in $L^\infty$. 
\begin{lemma}\label{borneLinfinity}
 Assume that $u\in M_{\Omega}$ achieves $\Sigma_{\nu}(\Omega)$ and $\alpha\le 1/4$. Then $u>0$. If we select $v$ as the multiple of $u$ that solves 
 \begin{align*}
\begin{cases}
\Delta^2 v -\Delta v+\alpha v=|v|^{\frac{8}{N-4}}v,&\mbox{in}\ \Omega,\\
\partial_{\nu} v =\partial_{\nu}(\Delta v)= 0,&\mbox{on}\ \partial\Omega,
\end{cases}
\end{align*} 
then $$\limsup_{\alpha \to 0} \|v\|_{L^{\infty}(\Omega)}<\infty.$$
 \end{lemma}

\begin{proof}
Observe first that we know from Lemma \ref{S_alphaK} and Lemma \ref{sharp-S-alpha} that $\Sigma_{\nu}(\Omega)$ is indeed achieved. When $\alpha$ is small enough (or the measure of $\Omega$ is small enough), this is in fact simpler to show since for $u_1=1$, we have
\begin{equation*}
\frac{\int_{\Omega}(|\Delta u_{1}|^2+|\nabla u_1|^2+\alpha |u_1|^2)\, dx}{\left(\int_{\Omega}|u_1|^{2^{*}}\, dx\right)^{\frac{2}{2^{*}}}}=\alpha |\Omega|^{4/N}< S/{2^{4/N}}.
\end{equation*}
It is by now standard to show that if $u$ changes sign, then $u$ cannot be a minimizer, see e.g.,  \cite{MR2833588,MR3855391}. We sketch the argument for completeness. We can write 
\begin{equation*}
\Sigma_{\nu}(\Omega) = \frac{\int_{\Omega}|-\Delta u + \frac12 u|^2\, dx + (\alpha-\frac14)\int_{\Omega} |u|^2\, dx}{\left(\int_{\Omega}|u|^{2^{*}}\, dx\right)^{\frac{2}{2^{*}}}}.
\end{equation*}
If $-\Delta(-) u + \frac12 (-)u\ge 0$ then $(-)u>0$ by the strong maximum principle (observe $u\not\equiv 0$). If not, take $v\in H^2_{\nu}(\Omega)$ to be the unique solution of 
$$-\Delta v + \frac12 v = 
|-\Delta u + \frac12 u|,\ x\in \Omega.$$
By the strong maximum principle, we infer that $v(x) > |u(x)|$ in $\Omega$ so that 
$$0\le \int_{\Omega}|-\Delta v + \frac12 v|^2\, dx + (\alpha-\frac14)\int_{\Omega} |v|^2\, dx < \int_{\Omega}|-\Delta u + \frac12 u|^2\, dx + (\alpha-\frac14)\int_{\Omega} |u|^2\, dx$$
and $$\int_{\Omega}|v|^{2^{*}}\, dx > \int_{\Omega}|u|^{2^{*}}\, dx.$$
This contradicts the fact that $u$ is a minimizer. 

Now, since $u$ is a minimizer and $u\in M_{\Omega}$, we have 
$$\int_{\Omega}(|\Delta u|^2+|\nabla u|^2+\alpha |u|^2)\, dx \le \alpha |\Omega|^{4/N}.$$
Take $\alpha_k\to 0$ and denote by $(u_k)_{k\in\N}$ a sequence of minimizers. Then 
$$\int_{\Omega}(|\Delta u_k|^2+|\nabla u_k|^2)\, dx\to 0\hspace{.3cm}\text{as }k\to\infty$$
and $(u_k)_{k\in\N}$ is bounded in $L^2(\Omega)$. Observe that $u_k$ solves the equation 
\begin{align*}
\begin{cases}
\Delta^2 u_k -\Delta u_k+\alpha_k u_k=\mu_k|u_k|^{\frac{8}{N-4}}u_k,&\mbox{in}\ \Omega,\\
\partial_{\nu} u_k =\partial_{\nu}(\Delta u_k)= 0,&\mbox{on}\ \partial\Omega,
\end{cases}
\end{align*}
where $\mu_k=\Sigma_{\nu,\alpha_k}(\Omega)\le \alpha_k |\Omega|^{4/N}$. Define $v_k=\mu_k^{\frac8{N-4}}u_k$ so that \eqref{vk} holds. Interior estimates show that $v_k$ is smooth. Clearly $(v_k)_{k\in\N}$ converges strongly to zero in $H^2(\Omega)$. To show that $\limsup_k\|v_k\|_{L^{\infty}(\Omega)}<\infty$, one can just borrow the blow-up argument of Gidas and Spruck (arguing therefore by contradiction) used in \cite[Lemma 2.1]{AdimYad93} by taking the blow-up profile 
\begin{equation*}
 w_k(y):=t^{\frac{N-4}{2}}_{k}v_k(x_k+t_k y),
\end{equation*}
where
$(x_k)_{k\in\N}\subset\overline{\Omega}$ is such that
\begin{equation*}
 M_k=v_k(x_k)=\|v_k\|_{L^{\infty}(\Omega)}
\end{equation*}
and 
\begin{equation*}
 M_k t^{\frac{N-4}{2}}_{k}=1.
\end{equation*}
\end{proof}


\begin{proof}[Proof of Theorem \ref{rigidity}]
Lemma \ref{borneLinfinity} combined with Lemma \ref{borneLinfinity0} imply any sequence of minimizers uniformly vanish as $\alpha\to 0$. 
Using Poincar\'e inequality and the uniform convergence to zero, one then shows that $u-\frac1{|\Omega|}\int_\Omega u\, dx=0$ when $\alpha$ is small enough. This is the original argument of Ni and Takagi \cite{MR849484}, see also \cite{MR3564729,AdimYad93}.
\end{proof}

\bibliographystyle{plain}
\bibliography{critical-reference}
\end{document}